\documentclass[a4paper,11pt]{amsart}
\usepackage{amsmath,amssymb,amsthm}
\usepackage[latin1]{inputenc}
\usepackage{version,tabularx,multicol}
\usepackage{graphicx,float,psfrag}
 \usepackage{stmaryrd}
 \usepackage{mathrsfs}
 \usepackage{dsfont}
 \usepackage{nicefrac}

\theoremstyle{plain}
\newtheorem{theo}{Theorem}[section]
\newtheorem*{theo*}{Theorem}

\newtheorem{prop}[theo]{Proposition}
\newtheorem{lem}[theo]{Lemma}
\newtheorem{defi}[theo]{Definition}
\newtheorem*{defi*}{Definition}
\theoremstyle{remark}
\newtheorem{rem}[theo]{Remark}

\newcommand{\N}{\mathbb{N}}
\newcommand{\R}{\mathbb{R}}
\newcommand{\Q}{\mathbb{Q}}

\newcommand{\E}{\mathbb{E}}
\renewcommand{\P}{\mathbb{P}}
\newcommand{\T}{\mathbb{T}}

\newcommand{\bt}{\mathbf{t}}

\newcommand{\ind}{\mathds{1}}

\title[Penalization of GW processes]{Penalization of Galton-Watson processes}



\date{\today}
\author{Romain Abraham} 
\author{Pierre Debs}
\address{
Institut Denis Poisson,
Universit\'e d'Orl\'eans,
Universit\'e de Tours,
CNRS,
Rue de Chartres,
B.P. 6759,
45067 Orl\'eans cedex 2,
France.
}
  
\email{romain.abraham@univ-orleans.fr} 
\email{pierre.debs@univ-orleans.fr}
\begin{document}

\begin{abstract}
We apply the penalization technique introduced by Roynette, Vallois, Yor for Brownian motion to Galton-Watson processes with a penalizing function of the form $P(x)s^x$ where $P$ is a polynomial of degree $p$ and $s\in[0,1]$. We prove that the limiting martingales obtained by this method are most of the time classical ones, except in the super-critical case for $s=1$ (or $s\to 1$) where we obtain new martingales. If we make a change of probability measure with this martingale, we obtain a multi-type Galton-Watson tree with $p$ distinguished infinite spines.
\end{abstract}

\maketitle

\section{Introduction}

Let $(Z_n)_{n\ge 0}$ be a Galton-Watson process (GW) associated with an offspring distribution $q=(q_n,n\in\N)$. We denote by $\mu$ the first moment of $q$ and recall that the process is said to be sub-critical (resp. critical, resp. super-critical) if $\mu<1$ (resp. $\mu=1$, resp. $\mu>1$) and that the process suffers a.s. extinction in the sub-critical and critial cases (unless the degenerate case $q_1=1$) whereas it has a positive probability $1-\kappa$ of survival in the super-critical case. Moreover, the constant $\kappa$ is the smallest positive fix point of the generating function $f$ of $q$. We refer to \cite{AN} for general results on GW processes.

It is easy to check that the two processes $(Z_n/\mu^n)_{n\ge 0}$ and $(\kappa^{Z_n-1})_{n\ge 0}$ are martingales with respect to the natural filtration $(\mathscr{F}_n)_{n\ge 0}$ associated with $(Z_n)_{n\ge 0}$, with mean 1. Moreover, given a martingale $(M_n)_{n\ge 0}$ with mean 1, we can define a new process $(\widetilde Z_n)_{n\ge 0}$ by a change of probability: for every nonnegative measurable functional $\varphi$, we have
\[
\E\left[\varphi(\widetilde Z_k,0\le k\le n)\right]=\E\left[M_n\varphi(Z_k,0\le k\le n)\right].
\]
The distribution of the process $\widetilde Z$ and of its genealogical tree is well-known for the two previous martingales. In the sub-critical or critical case, the process associated with the martingale $(Z_n/\mu^n)_{n\ge 0}$ is the so-called sized-biased GW and is a two-type GW. It can also be viewed as a version of the process conditioned on non-extinction, see \cite{Ke}. The associated genealogical tree is composed of an infinite spine on which are grafted trees distributed as the original one. In the super-critical case, if $\kappa\ne 0$, the process associated with the martingale $(\kappa^{Z_n-1})_{n\ge 0}$ is the original GW conditioned on extinction. It is a sub-critical GW with generating function $\tilde f(\cdot)=f(\kappa\,\cdot)/\kappa$ and mean $\tilde\mu=f'(\kappa)$. By combining these two results, we get a third martingale namely
\begin{equation}\label{eq:martingale}
M_n^{(1)}=\frac{Z_n\kappa^{Z_n-1}}{f'(\kappa)^n}
\end{equation}
and the associated process $\widetilde Z$ is distributed, if $0<\kappa<1$, as the size-biased process of the GW conditioned on extinction.

A general method called {\sl penalization} has been introduced by Roynette, Vallois and Yor \cite{RVY06a, RVY06b,RVY09} in the case of the one-dimensional Brownian motion to generate new martingales and to define, by a change of measure, Brownian-like processes conditioned on some specific zero-probability events. This method has also been used for similar problems applied to random walks, \cite{De09,De12}. It consists in our case in considering a function $\varphi(n,x)$ and studying the limit
\begin{equation}\label{eq:limit}
\lim_{m\to+\infty}\frac{\E\left[\ind_{\Lambda_n}\varphi(m+n,Z_{m+n})\right]}{\E[\varphi(m+n,Z_{m+n})]}
\end{equation}
with $\Lambda_n\in\mathscr{F}_n$. If this limit exists, it takes the form $\E[\ind_{\Lambda_n}M_n]$ where the process $(M_n)_{n\in\N}$ is a positive martingale with $M_0=1$ (see \cite{RY09} for more details). 

The study of conditioned GW goes back to the seminal work of Kesten \cite{Ke} and has recently received a renewed interest, see \cite{Ja,AD14a,AD14b}, mainly because of the possibility of getting other types of limiting trees than Kesten's. This work can also be viewed as part of this problem. For instance, penalizing by the martingale $Z_n/\mu^n$ prevents the process from extinction (this is the case considered in \cite{Ke}) whereas considering the weight $\kappa^{Z_n-1}$ penalizes the paths where the size of the population gets large.

In order to generalize the martingale \eqref{eq:martingale}, we first consider the function $\varphi(x)=H_p(x)s^x$ (that does not depend on $n$) for $0<s<1$ where $H_p$ denotes the $p$-th Hilbert's polynomial defined by
\begin{equation}\label{eq:Hilbert}
H_0(x)=1\mbox{ and }H_p(x)=\frac{1}{p!}\prod_{k=0}^{p-1}(x-k)\mbox{ for }p\ge 1.
\end{equation}
We prove that the limit \eqref{eq:limit} exists for every $s\in[0,1)$ but we always get already known limiting martingales. More precisely, see Theorems \ref{thm:s<1-q_0>0}, \ref{thm:s<1-q_0=0} and \ref{thm:lim-critique}, we have for every $p\in\N$, every $s\in[0,1)$, every $n\in\N$ and every $\Lambda_n\in\mathscr{F}_n$, 
\begin{itemize}
\item {\sl Critical and sub-critical case}.
\[
\lim_{m\to+\infty}\frac{\E\left[\ind_{\Lambda_n}H_p(Z_{m+n})s^{Z_{m+n}}\right]}{\E[H_p(Z_{m+n})s^{Z_{m+n}}]}=\begin{cases}
\E[\ind_{\Lambda_n}] &\mbox{if }p=0,\\
\E[Z_n/\mu^n\ind_{\Lambda_n}] & \mbox{if }p\ge 1.
\end{cases}
\]
This result also holds in the critical case for $s=1$.
\item {\sl Super-critical case}. We set $\mathfrak{a}=\min\{k\ge 0,\ q_k>0\}$. We have for every $p\ge 0$,
\[
\lim_{m\to+\infty}\frac{\E\left[\ind_{\Lambda_n}H_p(Z_{m+n})s^{Z_{m+n}}\right]}{\E[H_p(Z_{m+n})s^{Z_{m+n}}]}=\begin{cases}
\E\left[\kappa^{Z_n-1}\ind_{\Lambda_n}\right] & \mbox{if }p=0 \mbox{ and }\mathfrak{a}=0,\\
\E\left[\frac{Z_n\kappa^{Z_n-1}}{f'(\kappa)^n}\ind_{\Lambda_n}\right] &\mbox{if }p\ge 1\mbox{ and }\mathfrak{a}=0,\\
\E\left[\frac{1}{q_1^n}\ind_{Z_n=1}\ind_{\Lambda_n}\right] & \mbox{if }\mathfrak{a}=1,\\
\E\left[q_\mathfrak{a}^{-\frac{\mathfrak{a}^n-1}{\mathfrak{a}-1}}\ind_{Z_n=\mathfrak{a}^n}\ind_{\Lambda_n}\right] & \mbox{if }\mathfrak{a}\ge 2.
\end{cases}
\]
\end{itemize}

 Let us mention that the choice of the Hilbert's polynomials is only here to ease the computations but does not have any influence on the limit. Considering any polynomial of degree $p$ that vanishes at 0 leads to the same limit as for $H_p$.

A more interesting feature is to consider, in the super-critical case, $s=1$ or a sequence $s_n$ that tends to 1. It appears that the correct speed of convergence, in order to get non-trivial limits, leads to consider functions of the form
\[
\varphi_p(n,x)=H_p(x)e^{-ax/\mu^n}
\]
where $a\in\R_+$, see Theorem \ref{thm:limite_ratio}. We also describe the genealogical tree of $\widetilde Z$, see Theorem \ref{thm:distribution_p-spines}, which is the genealogical tree of a non-homogeneous multi-type GW (the offspring distribution of a node depends on its type and its generation). For the tree associated with the function $\varphi_p$, the types of the nodes run from 0 to $p$, the root being of type $p$. Moreover, the sum of the types of the offspring of one node is equal to the type of this node. Hence, nodes of type 0 give birth to nodes of type 0, nodes of type 1 give birth to one node of type 1 and nodes of type 0, nodes of type 2 give birth to either one node of type 2 or two nodes of type 1, and nodes of type 0, etc. For instance, the figure below gives some possible trees with a root of type 2 or 3. The type of the node is written in it, black nodes are of type 0.

\begin{center}
\begin{figure}[H]
\includegraphics[width=5cm]{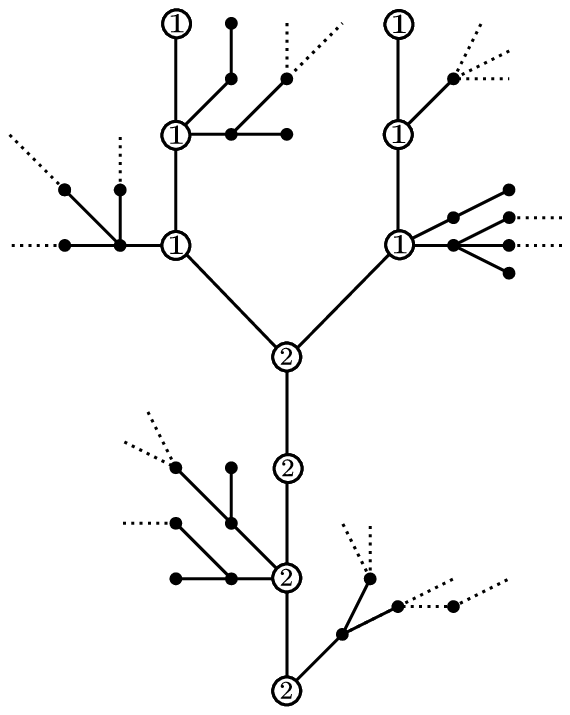}
\includegraphics[width=5cm]{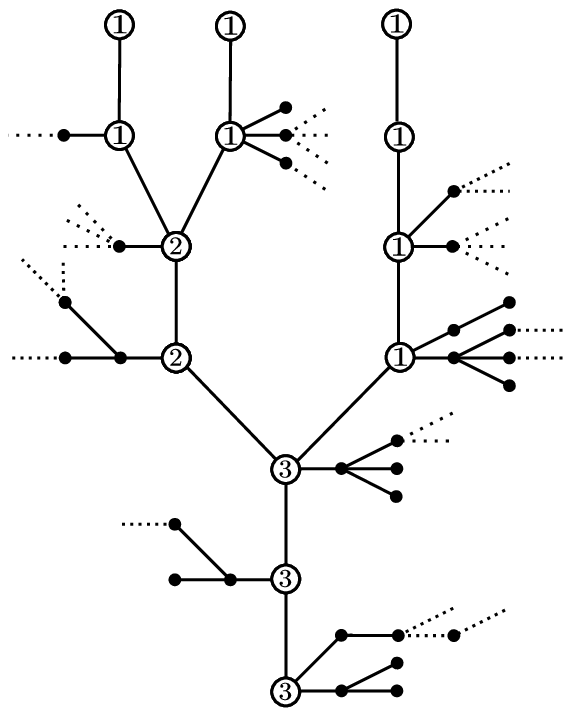}
\includegraphics[width=5cm]{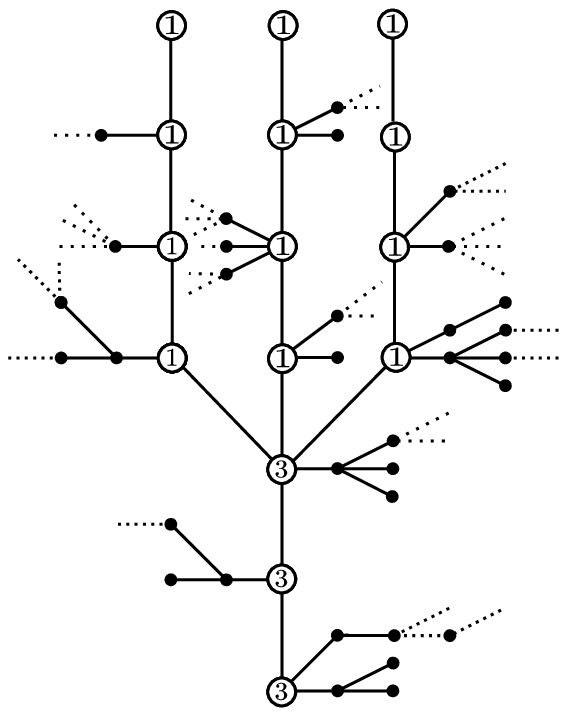}
\caption{Trees with a root of respective type 2, 3 and 3.}
\end{figure}
\end{center}

We see that, if the root is of type $p$, the tree exhibits a skeleton with $p$ infinite spines on which are grafted trees of type 0. The  distribution of such a tree is given in Definition \ref{def:biased-tree}. Let us mention that the trees of type 0 already appear in \cite{AD17} and may be infinite, the $p$-spines of the skeleton are not the only infinite spines of the tree. Multi-spine trees have already been considered, \cite{RSS,HR}, but they differ from those introduced here.

In the sub-critical case, if we suppose that there exists a second fix point $\kappa$ for the generating function $f$, the associated GW can be viewed as a super-critical GW conditioned on extinction and can be obtained from this super-critical GW by a standard change of measure. Then by combining the two changes of measure, the previous results can be used to get similar results in such a sub-critical case, see Theorem \ref{thm:sub-critical}.

The paper is organized as follows. In the second section, we introduce the formalism of discrete trees that we use in all the paper and define the distribution of Galton-Watson trees. In Section 3, we compute all the limits in the case $s\in[0,1)$. We then compute in Section 4 the limit in the super-critical case when $s_n\to 1$ and describe the distribution of the modified genealogical tree. We deduce the same kind of results in the sub-critical case in Section 4 and finish with an appendix that contains a technical lemma on Hilbert polynomials that we use in the proofs.

\section{Notations}

\subsection{The set of discrete trees}

Let $\mathscr{U}=\bigcup_{n=0}^{+\infty}(\N^*)^n$ be the set of finite sequences of positive integers with the convention $(\N^*)^0=\{\partial\}$. For every $u\in\mathscr{U}$, we set $|u|$ the length of $u$ i.e. the unique integer $n\in\N$ such that $u\in(\N^*)^n$. If $u$ and $v$ are two sequences of $\mathscr{U}$, we set $uv$ the concatenation of the two sequences with the convention $\partial u=u\partial=u$. For every $u\in\mathscr{U}\setminus\{\partial\}$, we define $\bar{u}$ the unique element of $\mathscr{U}$ such that $u=\bar{u}i$ for some $i\in\N^*$.

A tree $\bt$ rooted at $u\in\mathscr{U}$ is a subset of $\mathscr{U}$ that satisfies
\begin{itemize}
\item $u\in\bt$.
\item $\forall v\in \mathscr{U}$, $|v|<|u|\Longrightarrow v\not\in\bt$.
\item $\forall v\in\bt\setminus\{u\}$, $\bar v\in\bt$.
\item $\forall v\in\bt$, $\exists k_v(\bt)\in\N$, $\forall i\in\N^*$,\ $vi\in \bt\iff i\le k_v(\bt)$.
\end{itemize}

We denote by $\T_u$ the set of trees rooted at $u$ and by $\T=\bigcup_{u\in\mathscr{U}}\T_u$ the set of all trees.

For a tree $\bt\in\T$, we set $H(\bt)$ its height:
\[
H(\bt)=\max\{|u|,\ u\in\bt\},
\]
and we denote, for every $h\in\N^*$, by $\T^{(h)}$ (resp. $\T_u^{(h)}$) the subset of trees of $\T$ (resp. $\T_u$) with height less that $h$.

For every $n\in\N^*$, we denote by $\mathbf{1}_n=(1,\ldots,1)\in(\N^*)^n$, and we write for simplicity $\T_n$ (resp. $\T_n^{(h)}$) instead of $\T_{\mathbf{1}_n}$ (resp $\T^{(h)}_{\mathbf{1}_n}$).

For every $\bt\in\T$ and every $u\in \bt$, we set $\bt_u$ the subtree of $\bt$ rooted at $u$ i.e.
\[
\bt_u=\{v\in\bt,\ \exists w\in \mathscr{U},\ v=uw\}.
\]

For every $\bt\in\T$ and every $n\in\N$, we denote by $z_n(\bt)$ the number nodes of $\bt$ at height $n$:
\[
z_n(\bt)=\mbox{Card}(\{u\in\bt,\ |u|=n\}).
\]

For every $n\in\N^*$, we define on $\T$ the restriction operator $r_n$ by
\[
\forall \bt\in\T,\ r_n(\bt)=\{u\in\bt,\ |u|\le n\}.
\]
Classical results give that the distribution of a random tree $\tau$ on $\T$ is characterized by the family of probabilities $(\P(r_n(\tau)=\bt),\ n\in\N^*,\ \bt\in\T^{(n)})$.

\subsection{Galton-Watson trees}

Let $q=(q_n,\ n\in\N)$ be a probability distribution on the nonnegative integers. We set $\mu=\sum_{n=0}^{+\infty} nq_n$ its mean and always suppose that $\mu<+\infty$.

A $\T_\partial$-valued random tree $\tau$ is said to be Galton-Watson tree with offspring distribution $q$ under $\P$ if, for every $h\in\N^*$ and every $\bt\in\T_\partial^{(h)}$,
\[
\P(r_h(\tau)=\bt)=\prod_{u\in r_{h-1}(\bt)}q_{k_u(\bt)}.
\]

The generation-size process defined by $(Z_n=z_n(\tau),\ n\in\N)$ is the classical Galton-Watson process with offspring distribution $q$ starting with a single individual at time 0.

As we will later consider inhomogeneous Galton-Watson trees (whose offspring distribution depends on the height of the node), we define for every $k\in\N$ the distribution $\P_k$ under which the generation-size process is a Galton-Watson process starting with a single individual at time $k$:
\[
\forall h>k,\ \forall \bt\in\T_k^{(h)},\ \P_k(r_h(\tau)=\bt)=\prod_{u\in r_{h-1}(\bt)}q_{k_u(\bt)}.
\]
In other word, the random tree $\tau$ under $\P_k$ is distributed as $\mathbf{1}_k\tau$ under $\P$, and $\P$ is equal to $\P_\partial$.

Let $f$ denote the generating function of $q$ and for every $n\ge 1$, we set $f_n$ the $n$-fold iterate of $f$:
\[
f_1=f,\qquad \forall n\ge 1,\ f_{n+1}=f_n\circ f.
\]
Then $f_n$ is the generating function of the random variable $Z_n$ under $\P$.

We recall now the classical result on the extinction probability of the
Galton-Watson tree  and introduce some notations. We denote by
$\{H(\tau)<+\infty\}=\bigcup _{n\in \N} \{ Z_n=0\}$ the extinction
event and denote by $\kappa$ the extinction probability:
\begin{equation}\label{eq:def-kappa}
\kappa=\P(H(\tau)<+\infty).
\end{equation} 
Then, $\kappa$ is the
smallest non-negative root of $f(s)=s$. Moreover, we can prove that
\begin{equation}\label{eq:conv-kappa}
\forall s\in[0,1),\quad \lim_{n\to+\infty}f_n(s)=\kappa.
\end{equation}

We recall the  three following cases:
\begin{itemize}
\item The sub-critical case ($\mu<1$):  $\kappa=1$.
\item The critical case ($\mu=1$):  $\kappa=1$ (unless $q_1=1$ and then $\kappa=0$).
\item The super-critical case ($\mu>1$):  $\kappa\in [0, 1)$, the process has a
  positive probability  of non-extinction.
\end{itemize}
In the super-critical case, we recall that
\begin{equation}\label{eq:def-a}
\mathfrak{a}=\min\{k\ge 0,\ q_k>0\}
\end{equation}
and we say that we are in the Schroeder case if $\mathfrak{a}\le 1$ (which implies $f'(\kappa)>0$) and in the B\"otcher case if $\mathfrak{a}>1$ (in that case, we have $\kappa=f'(\kappa)=0$).

It is easy to check that the process $(Z_n/\mu^n,n\in\N)$ is a nonegative martingale under $\P$ and hence converges a.s. toward a random variable denoted by $W$.
Moreover, following \cite{Se}, we know that, in the super-critical case, if $q$ satisfies the so-called $L\log L$ condition i.e. $\E[Z_1\log Z_1]<+\infty$, then $W$ is non-degenerate and $\P(W=0)=\kappa$.

Let us denote by $\phi$ the Laplace transform of $W$. Then, $\phi$ is the unique (up to a linear change of variable) solution of Schroeder's equation (see \cite{Se}, Theorem 4.1):
\begin{equation}\label{eq:schroeder}
\forall a\ge 0,\quad f(\phi(a))=\phi(a\mu).
\end{equation}

\section{Standard limiting martingales}

In this section, we study the penalization function
\begin{equation}\label{eq:penal-s<1}
\varphi_p(x)=H_p(x)s^x
\end{equation}
for some fixed integer $p\in\N^*$ and some fixed $s\in[0,1)$ (or $s=1$ in the critical case).

\subsection{A formula for the conditional expectation}

Let $n,m$ be non-negative integers.
According to the branching property, conditionally on $\mathscr{F}_n$, we have
\[
Z_{n+m}\overset{(d)}{=}\sum_{j=1}^{Z_n}Z_m^{(j)}
\]
where the sequence $(Z^{(j)},j\ge 1)$ are i.i.d. copies of $Z$. Therefore we deduce that, for every $s\in[0,1)$, we have
\begin{equation}\label{eq:generatrice-cond}
\E[s^{Z_{m+n}}|\mathscr F_n]=\E\left[\prod_{j=1}^{Z_n}s^{Z_m^{(j)}}\biggm|\mathscr F_n\right]=\prod_{j=1}^{Z_n}\E\left[s^{Z_m^{(j)}}\right]=f_m(s)^{Z_n}.
\end{equation}

Let us denote by
\begin{equation}\label{eq:def-Sp}
S_{i,p}=\left\{(n_1,\ldots,n_i)\in(\N^*)^i,\ \sum_{k=1}^in_i=p\right\}.
\end{equation}
We have the following result:

\begin{lem}\label{lem:conditional}
Let $p\in\N^*$ and let $q$ be an offspring distribution with a finite $p$-th moment. For every $n,m\in\N$ and every $s\in[0,1)$, we have
\begin{equation}\label{eq:conditional}
\E\left[H_p(Z_{n+m})s^{Z_{n+m}-p}\Bigm|\mathscr{F}_n\right]=\sum_{i=1}^p H_i(Z_n)f_m(s)^{Z_n-i}\sum_{(n_1,\ldots,n_i)\in S_{i,p}}\prod_{j=1}^i\frac{f_m^{(n_j)}(s)}{n_j!}\cdot
\end{equation}
\end{lem}

\begin{proof}
First recall Fa\`a di Bruno's formula:
\begin{equation}\label{eq:faa-di-bruno}
\frac{d^p}{dx^p}f\bigl(g(x)\bigr)=p!\sum_{i=1}^p\frac{1}{i!}f^{(i)}\bigl(g(x)\bigr)\sum_{(n_1,\ldots,n_i)\in S_{i,p}}\prod_{j=1}^i\frac{g^{(n_j)}(x)}{n_j!}\cdot
\end{equation}
Using \eqref{eq:generatrice-cond}, we get
\begin{align*}
\E\left[H_p(Z_{n+m})s^{Z_{n+m}-p}|\mathscr F_n\right] & =\frac{1}{p!}\frac{d^p}{ds^p}\left(\E[s^{Z_{n+m}}|\mathscr F_n]\right)=\frac{1}{p!}\frac{d^p}{ds^p}\left(f_m(s)^{Z_n}\right)\\
&=\sum_{i=1}^p H_i(Z_n)f_m(s)^{Z_n-i}\sum_{(n_1,\ldots,n_i)\in S_{i,p}}\prod_{j=1}^i\frac{f_m^{(n_j)}(s)}{n_j!}
\end{align*}
\end{proof}

\subsection{The limiting martingale for $s\in[0,1)$ in the non-critical case}

\begin{lem}\label{lem:asymptotic_f}
Let $p\in\N^*$ and $q$ be a non-critical offspring distribution that satisfies the $L\,\log\,L$ condition. We suppose that we are in the Schroeder case if $q$ is super-critical. Then, there exists a positive function $C_p$, such that for all $s\in[0,1)$:
\begin{equation}\label{eq:deriv}
f_n^{(p)}(s)\underset{n\rightarrow+\infty}{\sim}C_p(s) \gamma^{n}
\end{equation}
where $\gamma=f'(\kappa)\in(0,1)$.
\end{lem}

Note that the $L\log L$ condition in the sub-critical case is needed to avoid $C_1\equiv 0$ (see  \cite{AN} pp. 38). 

\begin{proof}
The case $p=1$ is classical (with $C_1(s)=1$) and can be found  be found in \cite{AN} (pp. 38).
The rest of the proof is a generalisation of the case $p=2$ found in \cite{KMG}.

Assume that \eqref{eq:deriv} is true for all $j\leq p-1$.
Using again Fa\`a di Bruno's formula, we get:
\begin{align*}
\frac{f_{n+1}^{(p)}(s)}{f'_{n+1}(s)} &
=\frac{p!}{f'_{n+1}(s)}\sum_{i=1}^p \frac {1}{i!}f^{(i)}(f_n(s))\sum_{(n_1,\ldots,n_i)\in S_{i,p}} \prod _{j=1}^{i}{\frac {f_n^{(n_j)}(s)}{n_j!}}\\
&=\frac{p!}{f'_{n+1}(s)}\sum_{i=2}^{p} \frac {1}{i!}f^{(i)}(f_n(s))\sum_{(n_1,\ldots,n_i)\in S_{i,p}} \prod _{j=1}^{i}{\frac {f_n^{(n_j)}(s)}{n_j!}}\\
&\qquad +\frac{f'(f_n(s))f_{n}^{(p)}(s)}{f'_n(s)f'(f_n(s))}\cdot
\end{align*}

Therefore, we get
\begin{equation}\label{quotient}
\frac{f_{n+1}^{(p)}(s)}{f'_{n+1}(s)}-\frac{f_{n}^{(p)}(s)}{f'_n(s)}=\sum_{i=2}^{p} \frac {p!}{i!}f^{(i)}(f_n(s))\frac{1}{f'_{n+1}(s)}\sum_{(n_1,\ldots,n_i)\in S_{i,p}} \prod _{j=1}^{i}{\frac {f_n^{(n_j)}(s)}{n_j!}}
\end{equation}

For every $2\le i\le p$ and every $(n_1,\ldots,n_i)\in S_{i,p}$, as $n_j<p$ for every $j\le i$, we can use the induction hypothesis and deduce that there exists a positive constant $K$ such that 
\begin{equation}\label{geoconv}
\frac{1}{f'_{n+1}(s)} \prod _{j=1}^{i}{\frac {f_n^{(n_j)}(s)}{n_j!}}\underset{n\to+\infty}{\sim} K \gamma^{n(i-1)}.
\end{equation}
The continuity of $f^{(j)}$ and \eqref{eq:conv-kappa} imply that  $\lim_{n\rightarrow\infty}f^{(j)}(f_n(s))= f^{(j)}(\kappa)$ for all $j\geq 1$, thus formulas \eqref{quotient} and  \eqref{geoconv} implies that 
\[
\frac{f_{n+1}^{(p)}(s)}{f'_{n+1}(s)}-\frac{f_{n}^{(p)}(s)}{f'_n(s)}\underset{n\rightarrow+\infty}{\sim} K^\prime\gamma^n
\]
for some constant $K^\prime>0$.

As $\gamma\in(0,1)$, uniformly on any compact of $(0,1)$:
\[0<C_p(s):=\lim_{n\rightarrow+\infty}\frac{f^{(p)}_n(s)}{f'_n(s)}=\frac{f^{(p)}_1(s)}{f'_1(s)}+\sum_{n\geq1}\frac{f_{n+1}^{(p)}(s)}{f'_{n+1}(s)}-\frac{f^{(p)}_n(s)}{f'_n(s)}<+\infty\]
which is equivalent to $f^{(p)}_n(s)\underset{n\to+\infty}{\sim}C_p(s)f'_n(s)$. Applying again the lemma for $p=1$ gives the result.
\end{proof}

We can now state the main results concerning the limit of \eqref{eq:limit} with the penalization function \eqref{eq:penal-s<1}. We must separate two cases for super-critical offspring distributions depending on $q_0>0$ (which is equivalent to $\kappa >0$) or $q_0=0$ (which is equivalent to $\kappa=0$).

\begin{theo}\label{thm:s<1-q_0>0}
Let $p\in\N$ and let $q$ be a non-critical offspring distribution that admits a moment of order $p$ (and satisfies the $L\,\log\,L$ condition if $p=1$). We assume furthermore that $q_0>0$ (which is true if $q$ is sub-critical).
Then, for every $s\in[0,1)$, every $n\in\N$ and every $\Lambda_n\in\mathscr{F}_n$, we have
\begin{equation}
\lim_{m\to+\infty}\frac{\E\left[H_p(Z_{m+n})s^{Z_{m+n}}\ind_{\Lambda_n}\right]}{\E[H_p(Z_{m+n})s^{Z_{m+n}}]}=\E\left[\widetilde M_n^{(p)}\ind_{\Lambda_n}\right]
\end{equation}
with $\displaystyle \widetilde M_n^{(p)}=\begin{cases}
\kappa^{Z_n-1} & \mbox{if }p=0,\\
\frac{Z_n\kappa ^{Z_n-1}}{f'(\kappa)^n} & \mbox{if }p\ge 1.
\end{cases}$
\end{theo}

\begin{proof}
Let us first consider the case $p=0$. Using Equation \eqref{eq:generatrice-cond}, we get
\[
\frac{\E[s^{Z_{m+n}}\bigm|\mathscr{F}_n]}{\E[s^{Z_{m+n}}]}=\frac{f_m(s)^{Z_n}}{f_{m+n}(s)}\cdot
\]
As, for every $s\in [0,1)$, $\lim_{m\to+\infty}f_m(s)=\kappa>0$, we get
\[
\lim_{m\to+\infty}\frac{\E[s^{Z_{m+n}}\bigm|\mathscr{F}_n]}{\E[s^{Z_{m+n}}]}=\kappa^{Z_n-1}.
\]
Moreover, using the increasing property of $(f_n(s))_{n\ge1}$ in $s$ and $n$, we have
\[
0\le \frac{f_m(s)^{Z_n}}{f_{m+n}(s)}\le \frac{1}{q_0}
\]
so the dominated convergence theorem gives the result.% for $p=0$.

\medskip
Let us now suppose that $p\ge 1$. Applying \eqref{eq:deriv}, as $0<\gamma<1$, we get that, for every $i\ge 1$,
\[
\sum_{(n_1,\ldots,n_i)\in S_{i,p}} \prod _{j=1}^{i}{\frac {f_m^{(n_j)}(s)}{n_j!}}=C \gamma^{mi}+o(\gamma^{mi})
\]
for some constant $C$. Therefore, in \eqref{eq:conditional}, we get that the term for $i=1$ is dominant when $m\to+\infty$ and therefore
\[
\E\left[H_p(Z_{n+m})s^{Z_{n+m}-p}\Bigm|\mathscr{F}_n\right]=\frac{1}{p!}Z_nf_m(s)^{Z_n-1}f_m^{(p)}(s)+o(\gamma^m)=\frac{1}{p!}C_p(s)Z_n\kappa^{Z_n-1}\gamma^m+o(\gamma^m).
\]
Moreover, as for every $1\le i\le p$, we have
\[
H_i(Z_n)f_m(s)^{Z_n-i}=H_i(Z_n)f_m(s)^{Z_n-i}\ind_{Z_n\ge i}\le H_i(Z_n),
\] 
we have by dominated convergence
\[
\lim_{m\to+\infty}\E\left[H_i(Z_n)f_m(s)^{Z_n-i}\right]=\E\left[ H_i(Z_n)\kappa^{Z_n-i}\right],
\]
and by the same arguments as above, we get
\[
\forall \Lambda_n\in\mathscr{F}_n,\ \E\left[H_p(Z_{n+m})s^{Z_{n+m}-p}\ind_{\Lambda_n}\right]=\frac{1}{p!}C_p(s)\E\left[Z_n\kappa^{Z_n-1}\ind_{\Lambda_n}\right]\gamma^m+o(\gamma^m).
\]

Using \eqref{eq:conditional} with $n=0$ and \eqref{eq:deriv}, we have for $m\to+\infty$,
\[
\E\left[H_p(Z_{n+m})s^{Z_{n+m}-p}\right]=\frac{1}{p!}f_{n+m}^{(p)}(s)=\frac{1}{p!}C_p(s)\gamma^{n+m}+o(\gamma^{n+m}).
\]

Combining these two asymptotics yields
\[
\lim_{m\to+\infty}\frac{\E[H_p(Z_{m+n})s^{Z_{m+n}}\ind_{\Lambda_n}]}{\E[H_p(Z_{m+n})s^{Z_{m+n}}]}=\E\left[\frac{Z_n\kappa^{Z_n-1}}{\gamma^n}\ind_{\Lambda_n}\right]=\E\left[\frac{Z_n\kappa^{Z_n-1}}{f'(\kappa)^n}\ind_{\Lambda_n}\right].
\]
\end{proof}

\begin{rem}\label{rem:autre_poly}
Let $P$ be a polynomial of degree $p>0$ that vanishes at 0. Therefore, there exists constants $(\alpha_k)_{1\le k\le p}$ such that
\[
P=\sum_{k=1}^p\alpha_kH_k.
\]
Then, the previous asymptotics give, for every $n,m\in\N$ and every $\Lambda_n\in\mathscr{F}_n$,
\[
\E\left[P(Z_{m+n})s^{Z_{m+n}}\ind_{\Lambda_n}\right]=\left(\sum_{k=1}^p\alpha_k\frac{s^k}{k!}C_k(s)\right)\E\left[Z_n\kappa^{Z_n-1}\ind_{\Lambda_n}\right]\gamma^{m}+o(\gamma^m)
\]
and
\[
\E\left[P(Z_{m+n})s^{Z_{m+n}}\right]=\left(\sum_{k=1}^p\alpha_k\frac{s^k}{k!}C_k(s)\right)\gamma^{m+n}+o(\gamma^m)
\]
which implies that we obtain the same limit with $P$ or with $H_p$ in the penalizing function.
\end{rem}

Recall Definition \eqref{eq:def-a} of $\mathfrak{a}$.

\begin{theo}\label{thm:s<1-q_0=0}
Let $q$ be a super-critical offspring distribution that admits a moment of order $p\in\mathbb N$, and let us suppose that $\mathfrak{a}>0$ (or equivalently $q_0=0$). 
Then, for every $s\in(0,1)$, every $n\in\N$ and every $\Lambda_n\in\mathscr{F}_n$, we have
\begin{equation}\label{lim34}
\lim_{m\to+\infty}\frac{\E\left[H_p(Z_{m+n})s^{Z_{m+n}}\ind_{\Lambda_n}\right]}{\E[H_p(Z_{m+n})s^{Z_{m+n}}]}=\begin{cases}
\E\left[q_1^{-n}\ind_{Z_n=1}\ind_{\Lambda_n}\right] & \mbox{if }\mathfrak{a}=1,\\
\E\left[q_{\mathfrak{a}}^{-\frac{\mathfrak{a}^n-1}{\mathfrak{a}-1}}\ind_{Z_n=\mathfrak{a}^n}\ind_{\Lambda_n}\right] & \mbox{if }\mathfrak{a}\ge 2.
\end{cases}
\end{equation}
\end{theo}

\begin{proof}
The reasoning is similar to the previous one. 

Let us first consider the case $p=0$, $\mathfrak{a}=1$. In that case, we have (see \cite{AN} pp. 40 Corollary 1),
\begin{equation}\label{genequiv}
\forall s\in(0,1),\ f_m(s)\sim C_1(s)f'(0)^m=C_1(s)q_1^m
\end{equation}
for a positive function $C_1(s)$.
Therefore 
\[
\frac{\E[s^{Z_{m+n}}\bigm|\mathscr{F}_n]}{\E[s^{Z_{m+n}}]}=\frac{f_m(s)^{Z_n}}{f_{m+n}(s)}\sim q_1^{m(Z_n-1)}q_1^{-n}
\]
which converges to 0 if $Z_n>1$ and to $q_1^{-n}$ if $Z_n=1$. We conclude then by dominated convergence as  $\nicefrac{f_m(s)^{Z_n}}{f_{m+n}(s)}\le f_2(0)^{-1}$ for all $m\ge2$.% for $s\in[\delta,1)$ for any $\delta>0$.

\medskip
Let us now suppose that $p\ge 1$ and $\mathfrak{a}=1$. Using \eqref{eq:conditional}, Lemma \ref{lem:asymptotic_f} and \eqref{genequiv}, we have
\[
\E\left[H_p(Z_{n+m})s^{Z_{n+m}-p}\Bigm|\mathscr{F}_n\right]=\sum_{i=1}^nC_i H_i(Z_n)q_1^{mZ_n}(1+o(1))
\]
for some constants $C_i$ (note that $\gamma=f^{\prime }(0)=q_1$ here) and
\[
\E\left[H_p(Z_{n+m})s^{Z_{n+m}-p}\right]=\frac{1}{p!}f_{n+m}^{(p)}(s)=\frac{1}{p!}C_p(s)q_1^{n+m}+o(q_1^m)
\]
which yields for some constant $K>0$
\[
\frac{\E\left[H_p(Z_{n+m})s^{Z_{n+m}}\Bigm|\mathscr{F}_n\right]}{\E\left[H_p(Z_{n+m})s^{Z_{n+m}}\right]}\sim K q_1^{-n}q_1^{m(Z_n-1)}\sum_{i=1}^nC_i H_i(Z_n).
\]
This ratio tends to 0 if $Z_n>1$ and to $K^\prime q_1^{-n}$ otherwise, with $K^\prime>0$. Dominated convergence Theorem ensures the existence of the limit \eqref{lim34} and we can easily find that $K^\prime=1$ recalling that necessarily $(K^\prime q_1^{-n}\mathds{1}_{Z_n=1})_{n\ge 0}$ is a martingale with mean equals to 1.

For the case $\mathfrak{a}\ge 2$, we use the asymptotics given in the following lemma whose proof is postponed after the current proof.

\begin{lem}\label{lem:equiv_deriv}
For every $p\in\N$ and every $s\in(0,1)$, there exists a positive constant $K_p(s)$ such that
\[
f_m^{(p)}(s)=K_p(s)\mathfrak{a}^{mp}e^{\mathfrak{a}^mb(s)}\left(1+o(1)\right)
\]
where
\[
b(s)=\log s +\sum_{j=0}^{+\infty}\mathfrak{a}^{-j-1}\log \frac{f_{j+1}(s)}{f_j(s)^\mathfrak{a}},
\]
\end{lem}

In that case, we have for $p=0$ as $m\to +\infty$
\[
\frac{\E[s^{Z_{m+n}}\bigm|\mathscr{F}_n]}{\E[s^{Z_{m+n}}]}=\frac{f_m(s)^{Z_n}}{f_{m+n}(s)}\sim K_0(s)^{Z_n-1}e^{\mathfrak{a}^mb(s)(Z_n-\mathfrak{a}^n)}\underset{m\to+\infty}{\longrightarrow}K_0(s)^{\mathfrak{a}^n-1}\ind_{Z_n=\mathfrak{a}^n}
\]
since $b(s)<0$.
We conclude either by saying that $K_0(s)=q_\mathfrak{a}^{-1/\mathfrak{a}-1}$ by \cite{FW} Lemma 10, or by using the fact that the limit is a martingale with mean 1.

For $p\ge 1$, we use Lemma \ref{lem:equiv_deriv} to get that, for every $1\le i\le p$ and every $(n_1,\ldots,n_i)\in S_{i,p}$, we have as $m\to+\infty$,
\[
\prod_{j=1}^i\frac{f_m^{(n_j)}(s)}{n_j!}\sim \prod _{j=1}^iK_{n_j}(s)\mathfrak{a}^{m.n_j}e^{\mathfrak{a}^mb(s)}/n_j!=K(s)\mathfrak{a}^{mp}e^{\mathfrak{a}^mb(s)i}
\]
for some constant $K(s)$. Hence, we have for $1\le i\le p$,
\[
\sum _{(n_1,\ldots, n_i)\in S_{i,p}}\prod_{j=1}^i\frac{f_m^{(n_j)}(s)}{n_j!}\sim \widetilde K_i(s)\mathfrak{a}^{mp}e^{\mathfrak{a}^mb(s)i}
\]
for another constant $\widetilde K_i(s)$ since all the terms in the sum are nonnegative and of the same order.

Finally, using \eqref{eq:conditional}, we get
\begin{align*}
\E\left[H_p(Z_{n+m})s^{Z_{n+m}-p}\bigm|\mathscr{F}_n\right] & =\sum_{i=1}^pH_i(Z_n)f_m(s)^{Z_n-i}\sum _{(n_1,\ldots, n_i)\in S_{i,p}}\prod_{j=1}^i\frac{f_m^{(n_j)}(s)}{n_j!}\\
& \sim \sum_{i=1}^p K_0(s)^{Z_n-i}\widetilde K_i(s)H_i(Z_n)e^{\mathfrak{a}^m(Z_n-i)b(s)}\mathfrak{a}^{mp}e^{\mathfrak{a}^mb(s)i}\\
& =\hat K(s,Z_n) e^{\mathfrak{a}^mZ_nb(s)}\mathfrak{a}^{mp}
\end{align*}
for some function $\hat K$, again since all the terms of the sum are nonnegative and of the same order.

This gives
\[
\frac{\E\left[H_p(Z_{n+m})s^{Z_{n+m}-p}\bigm|\mathscr{F}_n\right]}{\E\left[H_p(Z_{n+m})s^{Z_{n+m}-p}\right]}
\sim \frac{\hat K(s,Z_n)}{K_p(s)}\mathfrak{a}^{-np}e^{\mathfrak{a}^mb(s)(Z_n-\mathfrak{a}^n)}\underset{m\to+\infty}\longrightarrow C_n\ind_{Z_n=\mathfrak{a^n}}
\]
where $C_n$ is a constant depending on $n$ that is computed again by saying that the limit is a martingale with mean 1.
\end{proof}

\begin{rem}
The same arguments as in Remark \ref{rem:autre_poly} can be used to show that the limit does not depend of the choice of the polynomial
\end{rem}

We now finish this section with the proof of Lemma \ref{lem:equiv_deriv}.

\begin{proof}[Proof of Lemma \ref{lem:equiv_deriv}]
In the proof, the letter $K$ will denote a constant that depends on $s$ and may change from line to line.

\medskip
Lemma 10 in \cite{FW} states that, for every $\delta>0$ and every $s\in(0,1-\delta)$, we have
\[
f_m(s)= q_\mathfrak{a}^{-1/(\mathfrak{a}-1)}e^{a^mb(s)}\left(1+o(e^{-\delta \mathfrak{a}^m})\right)
\]
which implies the result for $p=0$.

To prove the result for $p=1$, we follow the same ideas as in the proof of \cite{FW}, Lemma 10. We still consider $s\in(0,1-\delta)$ for some $\delta>0$. First, we have
\[
f'_{m+1}(s)=f'_m(s)f'(f_m(s))=f'_m(s)\sum_{k=0}^{+\infty}(\mathfrak{a}+k)q_{\mathfrak{a}+k}f_m(s)^{\mathfrak{a}+k-1},
\]
which gives
\begin{align*}
0\le \frac{f'_{m+1}(s)}{f'_m(s)\mathfrak{a}q_\mathfrak{a}f_m(s)^{\mathfrak{a}-1}}-1 & \le \sum_{k=1}^{+\infty}\frac{(\mathfrak{a}+k)q_{\mathfrak{a}+k}}{\mathfrak{a}q_\mathfrak{a}}f_m(s)\\
& \le K e^{\mathfrak{a}^mb(s)}
\end{align*}
by Lemma 13 of \cite{FW}. Therefore, as $\ln(1+u)\le u$ for every nonnegative $u$, we have
\[
0\le \ln\left(\frac{f'_{m+1}(s)}{f'_m(s)\mathfrak{a}q_\mathfrak{a}f_m(s)^{\mathfrak{a}-1}}\right) \le K e^{\mathfrak{a}^mb(s)}
\]
which implies that the series
\[
\sum_{n=0}^{+\infty}\ln\left(\frac{f'_{m+1}(s)}{f'_m(s)\mathfrak{a}q_\mathfrak{a}f_m(s)^{\mathfrak{a}-1}}\right)
\]
converges.
Using the asymptotics for $f_m(s)$ of Lemma 10 of \cite{FW} we get that
\[
\ln\left(\frac{f'_{m+1}(s)}{f'_m(s)\mathfrak{a}e^{\mathfrak{a}^m(\mathfrak{a}-1)b(s)}}\right)\sim \ln\left(\frac{f'_{m+1}(s)}{f'_m(s)\mathfrak{a}q_\mathfrak{a}f_m(s)^{\mathfrak{a}-1}}\right)
\]
and hence that the series
\[
\tilde b(s):=\sum_{m=0}^{+\infty}\ln\left(\frac{f'_{m+1}(s)}{f'_m(s)\mathfrak{a}e^{\mathfrak{a}^m(\mathfrak{a}-1)b(s)}}\right)
\]
converges.

Moreover, as
\[
\ln\frac{f'_m(s)}{\mathfrak{a}^me^{(a^m-1)b(s)}}=\tilde b(s)-\sum_{k={m}}^{+\infty}\ln\left(\frac{f'_{k+1}(s)}{f'_k(s)\mathfrak{a}e^{\mathfrak{a}^k(\mathfrak{a}-1)b(s)}}\right)=\tilde b(s)+o(1),
\]
we obtain
\[
\frac{f'_m(s)}{\mathfrak{a}^me^{(a^m-1)b(s)}}=e^{\tilde b(s)}(1+o(1))
\]
which is the looked after formula for $p=1$.

\medskip
We finish the proof  by induction on $p$ as for the proof of Lemma \ref{lem:asymptotic_f}. Let $p\ge 2$ and let us suppose that the asymptotics of Lemma \ref{lem:equiv_deriv} are true for every $j<p$. 
Recall Equation \eqref{quotient}
\begin{equation}\label{eq:recurrence2}
\frac{f_{m+1}^{(p)}(s)}{f'_{m+1}(s)}-\frac{f_{m}^{(p)}(s)}{f'_m(s)}=\sum_{i=2}^{p} \frac {p!}{i!}f^{(i)}(f_m(s))\frac{1}{f'_{m+1}(s)}\sum_{(n_1,\ldots,n_i)\in S_{i,p}} \prod _{j=1}^{i}{\frac {f_m^{(n_j)}(s)}{n_j!}}
\end{equation}
By the induction assumption, we have for every $1\le i\le p$, using the same computations as in the proof of Theorem \ref{thm:s<1-q_0=0},
\[
\sum_{(n_1,\ldots,n_i)\in S_{i,p}} \prod _{j=1}^{i}{\frac {f_m^{(n_j)}(s)}{n_j!}}\sim K \mathfrak{a}^{mp}e^{\mathfrak{a}^mb(s)i}.
\]
We also have
\[
f'_{m+1}(s)\sim K \mathfrak{a}^m e^{\mathfrak{a}^{m+1}b(s)}
\]
and for every $i\ge 2$,
\begin{align*}
f^{(i)}(f_m(s)) & = \begin{cases}
Kf_m(s)^{\mathfrak{a}-i}(1+o(1)) & \mbox{if }i<\mathfrak{a},\\
O(1) & \mbox{if }i\ge \mathfrak{a},
\end{cases}\\
& =\begin{cases}
Ke^{\mathfrak{a}^m(\mathfrak{a}-i)b(s)}(1+o(1)) & \mbox{if }i<\mathfrak{a},\\
O(1) & \mbox{if }i\ge \mathfrak{a}.
\end{cases}
\end{align*}

Hence, in the sum of \eqref{eq:recurrence2}, the terms for $i\le \mathfrak{a}$ (which exist since $\mathfrak{a}\ge 2$) are dominant and of order $K\mathfrak{a}^{m(p-1)}$.

We get
\[
\frac{f_{m+1}^{(p)}(s)}{f'_{m+1}(s)}-\frac{f_{m}^{(p)}(s)}{f'_m(s)}\sim K\mathfrak{a}^{m(p-1)}
\]
and, as the series diverge ($\mathfrak{a}^{p-1}>1$), the partial sums are also equivalent, which gives
\[
f_m^{(p)}(s)\sim K\mathfrak{a}^{m(p-1)}f'_m(s)\sim K\mathfrak{a}^{mp}e^{\mathfrak{a}^mb(s)}
\]
using the result for $p=1$.
 \end{proof}
 
\subsection{The limiting martingale for $s\in[0,1]$ in the critical case}

We finish with the result for a critical offspring distribution. As the arguments are the same as for the proof of Theorem \ref{thm:s<1-q_0>0}, we only give the main lines in the proof of the following theorem.

\begin{theo}\label{thm:lim-critique}
Let $q$ be a critical offspring distribution that admits a moment of order $p\in\N$. Then, for every $s\in[0,1]$, every $n\in\N$ and every $\Lambda_n\in\mathscr{F}_n$, we have
\[
\lim_{m\to+\infty}\frac{\E[H_p(Z_{m+n})s^{Z_{m+n}}\ind_{\Lambda_n}]}{\E[H_p(Z_{m+n})s^{Z_{m+n}}]}=\begin{cases}
\E[\ind_{\Lambda_n}] & \mbox{if }p=0,\\
\E[Z_n\ind_{\Lambda_n}] & \mbox{if }p\ge 1.
\end{cases}
\]
\end{theo}

\begin{proof}

We first study the case $s\in[0,1)$.

For $p=0$, the proof of Theorem \ref{thm:s<1-q_0>0} still applies with $\kappa=1$.

 For $p=1$, note that according to the dominated convergence theorem 
\begin{align*}
\frac{\E[Z_{n+m}s^{Z_{n+m}-1}|\mathscr F_n]}{f^{\prime}_{m+n}(s)}&=\frac{\E[Z_{n+m}s^{Z_{n+m}-1}|\mathscr F_n]}{\E[Z_{n+m}s^{Z_{n+m}-1}]}=\frac{Z_nf_m(s)^{Z_n-1}f^\prime_m(s)}{\E[Z_nf_m(s)^{Z_n-1}f^\prime_m(s)]}=\frac{Z_nf_m(s)^{Z_n-1}}{\E[Z_nf_m(s)^{Z_n-1}]}\\
&\underset{m\rightarrow+\infty}{\longrightarrow}\frac{Z_n}{\E[Z_n]}=Z_n
\end{align*}
giving our result. Moreover we can deduce from this limit's ratio that for all $n\ge0$, when $m$ goes to infinity
\begin{equation}\label{criticequiv}
f_{m+n}^{\prime}(s)\sim f^{\prime}_m(s).
\end{equation}

We then replace Lemma \ref{lem:asymptotic_f} by the following asymptotics for $f_n^{(p)}(s)$ whose proof is postponed at the end of the section.

\begin{lem}\label{lem:asymptotic_c}
In the critical case, for every $p\ge 1$, there exists a positive function $C_p$, such that for all $s\in[0,1)$:
\begin{equation}\label{eq:critic}
f_n^{(p)}(s)\underset{n\rightarrow+\infty}{\sim}C_p(s) f^\prime_n(s)
\end{equation}
\end{lem} 

The result then follows using the same arguments as in the proof of Theorem \ref{thm:s<1-q_0>0}.

\medskip
Let us now consider the case $s=1$. The case $p=0$ is trivial, so let us suppose that $p\ge 1$.

Equation \eqref{eq:conditional} applied to $s=1$ and $m=1$ gives
\[
\E[H_p(Z_{n+1})|\mathscr{F}_n]=\sum_{i=1}^pH_i(Z_n)\sum_{(n_1,\ldots,n_i)\in S_{i,p}}\prod_{j=1}^i\frac{1}{n_j!}\E[H_{n_j}(Z_1)]
\]
and an easy induction on $n$ and $p$ gives the following lemma.
\begin{lem}
Let $q$ be a critical offspring distribution that admits a moment of order $p$. Then there exists a polynomial $P$ of degree $p-1$ such that, for every $n\ge 0$,
\[
\E[H_p(Z_n)]=P(n).
\]
\end{lem}
This gives asympotics of $\E[H_p(Z_n)]$ of the form $cn^{p-1}$ as $n\to+\infty$. Plugging these asymptotics in \eqref{eq:conditional} and arguing as in the proof of Theorem \ref{thm:s<1-q_0>0} gives the result.

\end{proof}

\begin{proof}[Proof of Lemma \ref{lem:asymptotic_c}]
We first need to prove that $\sum_{n\geq0}f^{\prime}_n(s)<+\infty$.
Let $G$ be the function defined on $[0,1)$ by
\begin{align}
G(s):=\sum_{k\ge1}s^k\sum_{n\ge0}\P(Z_n=k)=\sum_{n\ge0}\sum_{k\ge1}s^{k}\P(Z_n=k)=\sum_{n\ge0}(f_n(s)-f_n(0)). 
\end{align}
According to \cite{KNS} pp.584, there exists a function $U$, such that for $s\in[0,1)$:
\[\lim_{n\rightarrow+\infty}n^2(f_n(s)-f_n(0))=U(s)<\infty\]
implying that $G$ is a power series that converges on $[0,1)$ and we have on this interval 
\[G^{\prime}(s)=\sum_{n\ge0}f^\prime_n(s)<+\infty.\]
The rest of the proof is very similar to the one of Lemma \ref{lem:asymptotic_f}: using \eqref{criticequiv} and the induction hypothesis, the equivalent of formula \eqref{geoconv} is
\begin{equation}
\frac{1}{f'_{n+1}(s)} \prod _{j=1}^{i}{\frac {f_n^{(n_j)}(s)}{n_j!}}\underset{n\to+\infty}{\sim} K (f^{\prime}_n(s))^{(i-1)}
\end{equation}
implying that
\[
\frac{f_{n+1}^{(p)}(s)}{f'_{n+1}(s)}-\frac{f_{n}^{(p)}(s)}{f'_n(s)}\underset{n\rightarrow+\infty}{\sim} K^\prime f^{\prime}_n(s)
\]
for some constant $K^\prime>0$. Consequently 
\[0<C_p(s):=\lim_{n\rightarrow+\infty}\frac{f^{(p)}_n(s)}{f'_n(s)}=\frac{f^{(p)}_1(s)}{f'_1(s)}+\sum_{n\geq1}\frac{f_{n+1}^{(p)}(s)}{f'_{n+1}(s)}-\frac{f^{(p)}_n(s)}{f'_n(s)}<+\infty\]
which is equivalent to $f^{(p)}_n(s)\underset{n\to+\infty}{\sim}C_p(s)f'_n(s)$. 
\end{proof}

\section{A new martingale in the super-critical case when $s\to 1$}

We are now considering the same penalization function \eqref{eq:penal-s<1} but with $s=1$ or with $s$ replaced by a sequence $(s_n)$ that tends to 1. More precisely, we consider functions of the form
\begin{equation}\label{eq:penal-s=1}
\varphi_p(n,x)=H_p(x)e^{-ax/\mu^n}
\end{equation}
for some non-negative constant $a$.

\subsection{The limiting martingale}

Recall that $H_p$ denotes the $p$-th Hilbert polynomial defined by \eqref{eq:Hilbert} and $\phi$ the Laplace transform of $W$ the limit of the martingale $(\nicefrac{Z_n}{\mu^n})_{n\ge0}$.   

For every $a\ge 0$, every $p\in\N^*$ and every $n\in\N$, we set, for every $x\in\R$,
\begin{equation}\label{eq:def-G}
G_n^{(p)}(x)=\frac{p!}{\phi^{(p)}(a)}\sum_{i=1}^p a_i^{(p)}(n)\phi(a/\mu^n)^{x-i}H_i(x)
\end{equation}
with
\begin{equation}\label{eq:def-a_i}
a_i^{(p)}(n)=\sum_{(n_1,\ldots,n_i)\in S_{i,p}}\prod_{r=1}^i\frac{\phi^{(n_r)}(a/\mu^n)}{n_r!}\cdot
\end{equation}

Let us first state the following relation between the coefficients $a_i^{(p)}(n)$ that will be used further.

\begin{lem}\label{lem:formule-a}
For every $i\in\N^*$ and every $(s_1,\ldots s_i)\in(\N^*)^i$, let us set $w=\sum_{j=1}^is_j$. Then, we have for every $n,p\ge 0$,
\[
\sum_{(\ell_1,\ldots,\ell_i)\in S_{i,p}}\prod_{j=1}^ia_{s_j}^{(\ell_j)}(n)=a_w^{(p)}(n)
\]
\end{lem}
\begin{proof}
Let us consider the polynomial
\[
P(X)=\sum_{k=1}^p\frac{\phi^{(k)}(a/\mu^n)}{k!}X^k.
\]
Then, by \eqref{eq:def-a_i}, $a_s^{(\ell)}(n)$ is the coefficient of order $\ell$ of the polynomial $P^s$ for every $s\le\ell\le p$.
The lemma is then just a consequence of the formula
\[
\prod _{j=1}^iP^{s_j}(X)=P^w(X).
\]
\end{proof}

\begin{theo}\label{thm:limite_ratio}
Let $p\in\N^*$. Let $q$ be a super-critical offspring distribution that admits a moment of order $p$. Then, for every $a\ge 0$, every $n\in\N$, and every $\Lambda_n\in\mathscr{F}_n$,
\[
\lim_{m\to+\infty}\frac{\E\left[H_p(Z_{m+n})e^{-aZ_{m+n}/\mu^{m+n}}\ind_{\Lambda_n}\right]}{\E\left[H_p(Z_{m+n})e^{-aZ_{m+n}/\mu^{m+n}}\right]}=\E\left[\frac{1}{\mu^{pn}}G_n^{(p)}(Z_n)\ind_{\Lambda_n}\right].
\]
\end{theo}

\begin{proof}
Let us first remark that for all $k\in\lbrace0,\dots,m\rbrace$
\begin{equation}
\frac{1}{\mu^{mk}}f_m^{(k)}\left(e^{-\frac{a}{\mu^{m+n}}}\right)
 =\E\left[e^{-a\frac{Z_m-k}{\mu^{n+m}}}\prod _{i=1}^k\frac{Z_m-i+1}{\mu^{m}}\right]
\underset{m\to+\infty}{\longrightarrow}(-1)^k\phi^{(k)}\left(\frac{a}{\mu^n}\right).\label{eq:limit-f_m}
\end{equation}

And, by the same argument, we have
\begin{align*}
\frac{1}{\mu^{p(m+n)}}\E\left[H_p(Z_{m+n})e^{-a\frac{Z_{m+n}}{\mu^{m+n}}}\right] & = \frac{1}{p!}\E\left[e^{-a\frac{Z_{m+n}}{\mu^{m+n}}}\prod_{i=1}^p\frac{Z_{m+n}-i+1}{\mu^{m+n}}\right]
\underset{m\to+\infty}{\longrightarrow}(-1)^p\frac{\phi^{(p)}(a)}{p!}\cdot
\end{align*}

Using Lemma \ref{lem:conditional} and \eqref{eq:limit-f_m}, we get
\begin{align*}
\frac{1}{\mu^{pm}} & \E\left[H_p(Z_{m+n})e^{-a\frac{Z_{m+n}}{\mu^{m+n}}}\Bigm| \mathscr{F}_n\right] \\
& =e^{-\frac{pa}{\mu^{m+n}}}\sum_{i=1}^p H_i(Z_n)f_m(e^{-\frac{a}{\mu^{m+n}}})^{Z_n-i}\sum_{(n_1,\ldots,n_i)\in S_{i,p}}\prod_{j=1}^i\frac{f_m^{(n_j)}(e^{-\frac{a}{\mu^{n+m}}})}{n_j!\,\mu^{n_jm}}\\
& \underset{m\to+\infty}{\longrightarrow}\sum_{i=1}^pH_i(Z_n)\phi(a/\mu^n)^{Z_n-i}\sum_{(n_1,\ldots,n_i)\in S_{i,p}}\prod_{j=1}^i(-1)^{n_j}\frac{\phi^{(n_j)}(a/\mu^n)}{n_j!}\\
& =(-1)^p\sum_{i=1}^pH_i(Z_n)\phi(a/\mu^n)^{Z_n-i}\sum_{(n_1,\ldots,n_i)\in S_{i,p}}\prod_{j=1}^i\frac{\phi^{(n_j)}(a/\mu^n)}{n_j!}\cdot
\end{align*}
Again, for every $1\le i\le p$, we have $H_i(Z_n)f_m(e^{-a/\mu^{m+n}})^{Z_n-i}\le H_i(Z_n)$ so, by dominated convergence, we get
\begin{multline*}
\lim_{n\to+\infty}\frac{\E\left[H_p(Z_{m+n})e^{-a\frac{Z_{m+n}}{\mu^{m+n}}}\ind_{\Lambda_n}\right]}{\E\left[H_p(Z_{m+n})e^{-a\frac{Z_{m+n}}{\mu^{m+n}}}\right]}\\
=\frac{1}{\mu^{pn}}\frac{p!}{\phi^{(p)}(a)}\sum_{i=1}^pH_i(Z_n)\phi(a/\mu^n)^{Z_n-i}\sum_{(n_1,\ldots,n_i)\in S_{i,p}}\prod_{j=1}^i\frac{\phi^{(n_j)}(a/\mu^n)}{n_j!}\cdot
\end{multline*}

\end{proof}

We end this subsection with the following uniqueness result concerning the limiting martingale in the homogeneous case i.e. $a=0$.

\begin{prop}\label{prop:unicite}
Let $p\ge 1$. There exists a unique polynomial $P_p$ of degree $p$ that vanishes at 0 such that the process $(X_n^{(p)})_{n\ge 0}$ defined by
\[
X_n^{(p)}=\frac{1}{\mu^{pn}}P_p(Z_n)
\]
is a martingale with mean 1.
\end{prop}

\begin{proof}

Existence is given by Theorem \ref{thm:limite_ratio}.

For uniqueness, let us write
\[
P_k=\sum_{i=1}^k c_i^{(k)}H_i(Z_n)
\]
and let us suppose that $X_n^{(k)}$ is a martingale with mean 1 for every $1\le k\le p$. This implies by taking the expectation that, for every $n\ge 0$ and every $k\le p$,
\begin{equation}\label{eq:expectation}
\sum_{i=1}^k \frac{1}{i!}c_i^{(k)}f_n^{(i)}(1)=\mu^{kn}.
\end{equation}
If we set for $1\le i,j\le p$
\[
f_{ij}=f_{i-1}^{(j)}(1),\quad c_{ij}=\begin{cases}c_i^{(j)} & \mbox{if }i\le j,\\ 0 & \mbox{if }i>j,\end{cases}\quad m_{ij}=\mu^{(i-1)j}
\]
and if we consider the square matrices of order $p$
\[
F=(f_{ij})_{1\le i,j\le p},\quad C=(c_{ij})_{1\le i,j\le p},\quad M=(m_{ij})_{1\le i,j\le p},
\]
Equations \eqref{eq:expectation} for $0\le n\le p-1$ write
\begin{equation}\label{eq:matrices}
FC=M
\end{equation}
where $C$ contains the unknown variables.

We know that, if $c_i^{(k)}=a_i^{(k)}$ where the $a_i^{(k)}$ are defined by \eqref{eq:def-a_i} with $a=0$ (and hence do not depend on $n$), $C$ is indeed a solution of Equation \eqref{eq:matrices} and is triangular with positive coefficients and hence $\det C\ne 0$. $M$ is a Vandermond matrix and hence also satisfies $\det M\ne 0$. Equation \eqref{eq:matrices} hence implies $\det F=\det M/\det C\ne 0$ which proves that $F$ is invertible and that \eqref{eq:matrices} has a unique solution.

\end{proof}

This proposition implies in particular that the choice of $H_p$ in Theorem 4.2 (if $a=0$) in the penalizing function is not relevant and any other polynomial of degree $p>1$ that vanishes at 0 gives the same limit.

\subsection{Distribution of the penalized tree}

In this section, we fix an integer $p\ge 0$ and consider an offspring distribution $q$ that admits a $p$-th moment (and that satisfies the $L\log L$ condtion if $p<2$). 

For every $n\ge n_0$, we consider the function
\[
G_{n,n_0}^{(p)}(x)=\begin{cases}
\displaystyle \frac{\phi(a/\mu^n)^x}{\phi(a/\mu^{n_0})} & \mbox{if }p=0\\
\displaystyle \frac{p!}{\phi^{(p)}(a/\mu^{n_0})}\sum_{i=1}^pa_i^{(p)}(n)H_i(x)\phi(a/\mu^n)^{x-i} & \mbox{if }p\ge 1
\end{cases}
\]
with $a_i^{(p)}(n)$ defined by \eqref{eq:def-a_i}
and we consider the martingale
\[
M_{n,n_0}^{(p)}=\frac{1}{\mu^{p(n-n_0)}}G_{n,n_0}^{(p)}(Z_n)
\]

We then define a new probability measure $\Q_{n_0}^p$ on $\T$ by
\begin{equation}\label{eq:def-Q}
\forall n\ge {n_0},\qquad \frac{d\Q_{n_0}^{(p)}}{d\P_{n_0}}_{|_{\mathscr{F}_n}}=M_{n,n_0}^{(p)}.
\end{equation}

We now define another probability measure $\mathbf{Q}_{n_0}^{(p)}$ on $\T$ as follows
\begin{defi}\label{def:biased-tree}
Under $\mathbf{Q}_{n_0}^{(p)}$, the random tree $\tau$ is distributed as an inhomogeneous multi-type Galton-Watson tree as follows
\begin{itemize}
\item The types of the nodes run from $0$ to $p$.
\item The root of $ \tau$ is of type $p$ and starts at height $n_0$.
\item A node of type $\ell$ at height $n$ gives, independently of the other nodes, $k$ offspring with respective types $(\ell_1,\ldots,\ell_k)$ such that $\ell_1+\cdots+\ell_k=\ell$ with probability
\[
q_k\frac{1}{\mu^\ell}\frac{\ell!}{\phi^{(\ell)}(a/\mu^n)}\prod_{j=1}^k\frac{\phi^{(\ell_j)}(a/\mu^{n+1})}{\ell_j!}\cdot
\]
\end{itemize}
\end{defi}

\begin{rem}
A node of type 0 at height $n$ gives $k$ offspring with probability
\[
q_k^0(n)=q_k\frac{\phi(a/\mu^{n+1})^k}{\phi(a/\mu^n)},
\]
all of them being of type 0.

Remark also that $q_k^0(n)=q_k$ if $a=0$.
\end{rem}

\begin{rem}
If a node is of type $\ell>0$, the condition $\ell_1+\cdots+\ell_k=\ell$ implies that this node has at least one offspring with non-zero type.
\end{rem}

\begin{rem}
The last property also writes: A node of type $\ell$ at height $n$ gives, independently of the other nodes, $k$ offspring, $k-i$ being of type 0, and $i$ of respective types $(\ell_1,\ldots,\ell_i)\in S_{i,\ell}$, with probability
\begin{equation}\label{eq:def-proba}
q_k\frac{\ell!}{\mu^\ell}\frac{\phi(a/\mu^{n+1})^{k-i}}{\phi^{(\ell)}(a/\mu^n)}\binom{k}{i}\prod_{j=1}^i\frac{\phi^{(\ell_j)}(a/\mu^{n+1})}{\ell_j!}\cdot
\end{equation}
The $i$ nodes with non-zero types are uniformly chosen among the $k$ offspring.

This equivalent formulation will be used in all the next proofs.
\end{rem}

\begin{lem}
Equation \eqref{eq:def-proba} indeed defines a probability distribution.
\end{lem}

\begin{proof}
We must prove that
\[
\sum_{k=1}^{+\infty}\sum_{i=1}^{k\wedge \ell}\sum_{(\ell_1,\dots,\ell_i)\in S_{i,\ell}}q_k\frac{\ell!}{\mu^\ell}\frac{\phi(a/\mu^{n+1})^{k-i}}{\phi^{(\ell)}(a/\mu^n)}\binom{k}{i}\prod_{j=1}^i\frac{\phi^{(\ell_j)}(a/\mu^{n+1})}{\ell_j!}=1.
\]

First remark that formula \eqref{eq:def-a_i} gives:
\[
\sum_{(\ell_1,\dots,\ell_i)\in S_{i,\ell}}\prod_{j=1}^i\frac{\phi^{(\ell_j)}(a/\mu^{n+1})}{\ell_j!}=a_i^{(\ell)}(n+1).
\]

Now, as $M_{n+1,n}^{(\ell)}$ is a martingale with mean one, we have by taking the expectation
\begin{multline*}
\frac{\ell!}{\mu^\ell \phi^{(\ell)}(a/\mu^n)}\sum_{i=1}^\ell a_i^{(\ell)}(n+1)\E_n\left[H_i(Z_{n+1})\phi(a/\mu^{n+1})^{Z_{n+1}-i}\right]=1\\
\iff \frac{\ell!}{\mu^\ell \phi^{(\ell)}(a/\mu^n)}\sum_{i=1}^\ell a_i^{(\ell)}(n+1)\sum_{k=i}^{+\infty}q_kH_i(k)\phi(a/\mu^{n+1})^{k-i}=1,
\end{multline*}
which ends the proof by inverting the sums and noting that $H_i(k)=\binom{k}{i}$.
\end{proof}

\begin{theo}\label{thm:distribution_p-spines}
For every $n_0\ge 0$ the probability measures $\mathbf{Q}_{n_0}^{(p)}$ end $\Q_{n_0}^{(p)}$ coincide.
\end{theo}

\begin{proof}
To prove the theorem, it suffices to prove that,
\begin{equation}\label{eq:recurrence}
\forall n\ge n_0,\ \forall \bt\in\T_{n_0}^n,\ \mathbf{Q}_{n_0}^{(p)}(r_n(\tau)=\bt)=\Q_{n_0}^{(p)}(r_n(\tau)=\bt)
\end{equation}
We prove this formula by induction on $p$.

\medskip
For $p=0$, we have, for every $n>n_0$ (the case $n=n_0$ is trivial as the tree $r_n(\tau)$ is reduced to the root),
\begin{align*}
\mathbf{Q}_{n_0}^{(0)}(r_n(\tau)=\bt) & =\prod_{r=n_0}^{n-1}\prod _{\{u\in\bt, |u|=r\}}q_{k_u(\bt)}^0(r) =\prod_{r=n_0}^{n-1}\prod _{\{u\in\bt, |u|=r\}} q_{k_u(\bt)}\frac{\phi(a/\mu^{r+1})^{k_u(\bt)}}{\phi(a/\mu^r)}\\
& =\left(\prod_{r=n_0}^{n-1}\frac{\phi(a/\mu^{r+1})^{z_{r+1}(\bt)}}{\phi(a/\mu^r)^{z_r(\bt)}}\right)\P_{n_0}(r_{n}(\tau)=\bt)\\
& =\frac{\phi(a/\mu^n)^{z_n(\bt)}}{\phi(a/\mu^{n_0})}\P_{n_0}(r_{n}(\tau)=\bt) =\Q_{n_0}^{(0)}(r_{n}(\tau)=\bt)
\end{align*}
since $z_{n_0}(\bt)=1$.

Let us now suppose that \eqref{eq:recurrence} is true for every $p'<p$. We prove that the property is true at rank $p$ by induction on $n$.

We have already mentioned that the formula is trivially true for $n=n_0$.

Let us now fix $n>n_0$ and let us suppose that the formula is true at rank $p$ for every $n'<n$ and let us prove it for $n$. Let $\bt\in \T_{n_0}^{(n)}$ and let us denote by $k_0$ the number of offspring of the root of $\bt$. We denote by $\bt_1,\ldots,\bt_{k_0}$ the (ordered) sub-trees of $\bt$ above the first generation. By decomposing according to the offspring of the root, we have
\begin{multline*}
\mathbf{Q}_{n_0}^{(p)}(r_n(\tau)=\bt) 
= \frac{p!}{\mu^p}q_{k_0}\sum_{i=1}^{k_0\wedge p}\binom{k_0}{i}\frac{\phi(a/\mu^{n_0+1})^{k_0-i}}{\phi^{(p)}(a/\mu^{n_0})}\sum_{(\ell_1,\ldots,\ell_i)\in S_{i,p}}\left(\prod_{j=1}^i\frac{\phi^{(\ell_j)}(a/\mu^{n_0+1})}{\ell_j!}\right)\\
\times \frac{1}{\binom{k_0}{i}}\sum_{1\le r_1<\cdots<r_i\le k_0}\left(\prod_{j=1}^i\mathbf{Q}_{n_0+1}^{(\ell_j)}(r_n(\tau)=\bt_{r_j})\right)
\times\left(\prod_{\substack{1\le k\le k_0\\ k\not\in\{r_1,\ldots,r_i\}}}\mathbf{Q}_{n_0+1}^{(0)}(r_n(\tau)=\bt_k)\right).
\end{multline*}

Therefore, as
\[
\P_{n_0}(r_n(\tau)=\bt)=q_{k_0}\prod_{j=1}^{k_0}\P_{n_0+1}(r_{n}(\tau)=\bt_j),
\]
we have
\begin{align*}
\frac{\mu^{p(n-n_0)}\mathbf{Q}_{n_0}^{(p)}(r_n(\tau)=\bt)}{\P_{n_0}(r_n(\tau=t))}& = p!\sum_{i=1}^{k_0\wedge p}\frac{\phi(a/\mu^{n_0+1})^{k_0-i}}{\phi^{(p)}(a/\mu^{n_0})}\sum_{(\ell_1,\ldots,\ell_i)\in S_{i,p}}\left(\prod_{j=1}^i\frac{\phi^{(\ell_j)}(a/\mu^{n_0+1})}{\ell_j!}\right)\\
&\qquad\qquad\times \sum_{1\le r_1<\cdots<r_i\le k_0}\left(\prod_{j=1}^i\frac{\mathbf{Q}_{n_0+1}^{(\ell_j)}(r_n(\tau)=\bt_{r_j})}{\P_{n_0+1}(r_n(\tau)=\bt_{\ell_j})}\mu^{\ell_j(n-n_0-1)}\right)\\
&\qquad\qquad\qquad\times\left(\prod_{\substack{1\le k\le k_0\\ k\not\in\{r_1,\ldots,r_i\}}}\frac{\mathbf{Q}_{n_0+1}^{(0)}(r_n(\tau)=\bt_k)}{\P_{n_0+1}(r_n(\tau)=\bt_k)}\right)\\
&\qquad = p!\sum_{i=1}^{k_0\wedge p}\frac{\phi(a/\mu^{n_0+1})^{k_0-i}}{\phi^{(p)}(a/\mu^{n_0})}\sum_{(\ell_1,\ldots,\ell_i)\in S_{i,p}}\left(\prod_{j=1}^i\frac{\phi^{(\ell_j)}(a/\mu^{n_0+1})}{\ell_j!}\right)\\
&\qquad\qquad\times \sum_{1\le r_1<\cdots<r_i\le k_0}\left(\prod_{j=1}^i\frac{\Q_{n_0+1}^{(\ell_j)}(r_n(\tau)=\bt_{r_j})}{\P_{n_0+1}(r_n(\tau)=\bt_{\ell_j})}\mu^{\ell_j(n-n_0-1)}\right)\\
&\qquad\qquad\qquad\times\left(\prod_{\substack{1\le k\le k_0\\ k\not\in\{r_1,\ldots,r_i\}}}\frac{\Q_{n_0+1}^{(0)}(r_n(\tau)=\bt_k)}{\P_{n_0+1}(r_n(\tau)=\bt_k)}\right).
\end{align*}
by the induction assumption on $p$ for $\ell_j<p$ (i.e. $i\ne 1$) and the induction assumption on $n$ for $\ell_j=p$ (i.e. $i=1$). By the definition of the measure $\Q^{(k)}_{n_0}$, we have
\begin{align*}
&\frac{\mu^{p(n-n_0)}\mathbf{Q}_{n_0}^{(p)}(r_n(\tau)=\bt)}{\P_{n_0}(r_n(\tau)=t)}  = p!\sum_{i=1}^{k_0\wedge p}\frac{\phi(a/\mu^{n_0+1})^{k_0-i}}{\phi^{(p)}(a/\mu^{n_0})}\sum_{(\ell_1,\ldots,\ell_i)\in S_{i,p}}\left(\prod_{j=1}^i\frac{\phi^{(\ell_j)}(a/\mu^{n_0+1})}{\ell_j!}\right)\\
&\times \sum_{1\le r_1<\cdots<r_i\le k_0}\left(\prod_{j=1}^i\frac{\ell_j!}{\phi^{(\ell_j)}(a/\mu^{n_0+1})}\sum_{s=1}^{\ell_j}a_s^{(\ell_j)}(n)\phi(a/\mu^n)^{z_{n}(\bt_{r_j})-s}H_s(z_n(\bt_{r_j}))\right)\\
& \qquad \qquad \times\left(\prod_{\substack{1\le k\le k_0\\ k\not\in\{r_1,\ldots,r_i\}}}\frac{\phi(a/\mu^n)^{z_n(\bt_k)}}{\phi(a/\mu^{n_0+1})}\right)\\
& \qquad =\frac{p!}{\phi^{(p)}(a/\mu^{n_0})}\sum_{i=1}^{k_0\wedge p}\sum_{(\ell_1,\ldots,\ell_i)\in S_{i,p}}\sum_{1\le r_1<\cdots<r_i\le k_0}\phi(a/\mu^n)^{z_n(\bt)}\\
& \qquad\qquad\times \prod_{j=1}^i\sum_{s=1}^{\ell_j}a_s^{(\ell_j)}(n)\phi(a/\mu^n)^{-s}H_s(z_n(\bt_{r_j}))\\
\intertext{using that $\displaystyle \sum_{k=1}^{k_0}z_n(\bt_k)=z_n(\bt)$}\\
& \qquad = \frac{p!}{\phi^{(p)}(a/\mu^{n_0})}\sum_{i=1}^{k_0\wedge p}\sum_{(\ell_1,\ldots,\ell_i)\in S_{i,p}}\sum_{1\le r_1<\cdots<r_i\le k_0}\phi(a/\mu^n)^{z_n(\bt)}\\
& \qquad\qquad \times \sum_{w=i}^p\sum_{(s_1,\ldots,s_i)\in S_{i,w}^+}\prod_{j=1}^ia_{s_j}^{(\ell_j)}(n)\phi(a/\mu^n)^{-s_j}H_{s_j}(z_n(\bt_{r_j}))\\
& \qquad = \frac{p!}{\phi^{(p)}(a/\mu^{n_0})}\sum_{w=1}^p\sum_{i=1}^{k_0\wedge w}\sum_{1\le r_1<\cdots<r_i\le k_0}\sum_{(s_1,\ldots,s_i)\in S_{i,w}}\phi(a/\mu^n)^{z_n(\bt)-w}\\
& \qquad\qquad\times \left(\sum_{(\ell_1,\ldots,\ell_i)\in S_{i,p}}\prod_{j=1}^ia_{s_j}^{(\ell_j)}(n)\right)\prod_{j=1}^iH_{s_j}(z_n(\bt_{r_j})).
\end{align*}

Using successively Lemma \ref{lem:formule-a} and Lemma \ref{lem:somme-Hk} gives
\begin{align*}
& \frac{\mu^{p(n-n_0)}\mathbf{Q}_{n_0}^{(p)}(r_n(\tau)=\bt)}{\P_{n_0}(r_n(\tau=t))}\\
& \qquad = \frac{p!}{\phi^{(p)}(a/\mu^{n_0})}\sum_{w=1}^pa_w^{(p)}(n)\phi(a/\mu^n)^{z_n(\bt)-w}\\
& \qquad\qquad \times \left(\sum_{i=1}^{k_0\wedge w}\sum_{1\le r_1<\cdots<r_i\le k_0}\sum_{(s_1,\ldots,s_i)\in S_{i,w}^+}\prod_{j=1}^iH_{s_j}(z_n(\bt_{r_j}))\right)\\
& \qquad =\frac{p!}{\phi^{(p)}(a/\mu^{n_0})}\sum_{w=1}^pa_w^{(p)}(n)\phi(a/\mu^n)^{z_n(\bt)-w}H_w(z_n(\bt))
\end{align*}
which ends the induction.
\end{proof}

\section{The sub-critical case}

%In this section, we fix a positive integer $p$ and consider a subcritical offspring distribution $q$ that admits a $p$-th moment.
%
%The main result is the following ratio convergence:
%
%\begin{theo}
%For every $a\ge 0$, we have
%\[
%\lim_{m\to+\infty}\frac{\E[H_p(Z_n)e^{-aZ_{m+n}/\mu^{m+n}}\bigm| \mathscr{F}_n]}{\E[H_p(Z_n)e^{-aZ_{m+n}/\mu^{m+n}}]}=\frac{Z_n}{\mu^n}\cdot
%\]
%\end{theo}
%
%\begin{proof}
%The ideas of this proof are the same as for the proof of Theorem \ref{thm:s<1}.
%
%By Equation \eqref{eq:conditional}, we have
%\begin{multline*}
%\E[H_p(Z_n)e^{-aZ_{m+n}/\mu^{m+n}}\bigm| \mathscr{F}_n]=\\
%e^{ap/\mu^{m+n}}\sum_{i=1}^p H_i(Z_n)f_m(e^{-a/\mu^{m+n}})^{Z_n-i}\sum_{(n_1,\ldots,n_i)\in S_{i,p}}\prod_{j=1}^i\frac{f_m^{(n_j)}(e^{-a/\mu^{m+n}})}{n_j!}\cdot
%\end{multline*}
%As we are in the subcritical case, we have
%\[
%\lim_{m\to+\infty}f_m(e^{-a/\mu^{m+n}})=1.
%\]
%Using Lemma \ref{lem:asymptotic_f}, as the function $C_p$ is continuous, we have
%\[
%f_m^{(k)}(e^{-a/\mu^{m+n}})\underset{m\to+\infty}{\sim}C_p(0)\gamma^k
%\]
%with $\gamma=f'(1)=\mu$.
%
%Therefore, we have
%\[
%\E[H_p(Z_n)e^{-aZ_{m+n}/\mu^{m+n}}\bigm| \mathscr{F}_n]\underset{m\to+\infty}{\sim}e^{ap/\mu^{m+n}}Z_n\mu^{Z_n-1}\frac{C_p(0)}{p!}\cdot
%\]
%
%Using the same kinf od arguments yields
%\[
%\E[H_p(Z_n)e^{-aZ_{m+n}/\mu^{m+n}}]\underset{m\to+\infty}{\sim}e^{ap/\mu^{m+n}}\frac{C_p(0)}{p!}
%\]
%and the limit of the ratio follows.
%
%\end{proof}

In this section, we consider a sub-critical offspring distribution $q$ and we assume that there exists $\kappa>1$ such that $f(\kappa)=\kappa$ and $f'(\kappa)<+\infty$ (this implies in particular that $q$ admits moments of any order).

We define $\bar f(t)=f(\kappa t)/\kappa$ for $t\in[0,1]$ and note that $\bar f$ is the generating function of a super-critical offspring distribution $\bar q$ with $\bar q_n=\kappa^{n-1}q_n$. The mean $\bar \mu$ of $\bar q$ is $f'(\kappa)$, the smallest positive fixed point of $\bar f$ is $\bar\kappa=1/\kappa$ and $\bar f'(\bar \kappa)=\mu$. Let $\bar \tau$ be the corresping genealogical tree. It is elementary to check that, for every $n\in\N$ and nonnegative measurable function $\varphi$, we have
\begin{equation}\label{eq:densite}
\E[\varphi(r_n(\bar \tau))]=\E\left[\kappa^{Z_n-1}\varphi(r_n(\tau))\right].
\end{equation}

We deduce from Theorem \ref{thm:s<1-q_0>0}, Theorem \ref{thm:limite_ratio} and \ref{thm:distribution_p-spines} the following result:

\begin{theo}\label{thm:sub-critical}
Let $p\in\N$. Let $q$ be a sub-critical offspring distribution with generating function $f$ and suppose that there exists a unique $\kappa>1$ such that $f(\kappa)=\kappa$ and $f^{(p\vee 1)}(\kappa)<+\infty$. Then for every every $n\in\N$ and every $\Lambda_n\in\mathscr{F}_n$, we have
\begin{itemize}
\item For every $s\in[0,\kappa)$,
\[
\lim_{m\to+\infty}\frac{\E\left[H_p(Z_{m+n})s^{Z_{m+n}}\ind_{\Lambda_n}\right]}{\E\left[H_p(Z_{m+n})s^{Z_{m+n}}\right]}=\begin{cases}
\E\left[\ind_{\Lambda_n}\right] & \mbox{if }p=0,\\
\E\left[\frac{Z_n}{\mu^n}\ind_{\Lambda_n}\right] & \mbox{if }p\ge 1.
\end{cases}
\]
\item For every $a\ge 0$,
\[
\lim_{m\to+\infty}\frac{\E\left[H_p(Z_{m+n})\kappa^{Z_{m+n}}e^{-a\frac{Z_{m+n}}{f'(\kappa)^{m+n}}}\ind_{\Lambda_n}\right]}{\E\left[H_p(Z_{m+n})\kappa^{Z_{m+n}}e^{-a\frac{Z_{m+n}}{f'(\kappa)^{m+n}}}\right]}=\E\left[\frac{1}{f'(\kappa)^{pn}}\kappa^{Z_n-1}\bar G_n^{(p)}(Z_n)\right]:=\E\left[\bar M_n^{(p)}\ind_{\Lambda_n}\right]
\]
where $\bar G_n^{(p)}$ is the function defined by \eqref{eq:def-G} associated with the offspring distribution $\bar q$.

Moreover, the probability measure $\Q_{n_0}^{(p)}$ defined by \eqref{eq:def-Q}
with $M^{(p)}$ replaced by $\bar M^{(p)}$, is the probability measure $\bar{\mathbf{Q}}_{n_0}^{(p)}$ defined in Definition \ref{def:biased-tree} with $q$ replaced by $\bar q$.
\end{itemize}
\end{theo}

\begin{proof}
We only prove one case, the other ones can be handled in the same way.

Let us consider $p\ge 1$ and $s\in[0,\kappa)$. Using Equation \eqref{eq:densite} then Theorem \ref{thm:s<1-q_0>0} (remark that, as $q$ is sub-critical, $q_0>0$), and then Equation \eqref{eq:densite} again, we have
\begin{align*}
\lim_{m\to+\infty}\frac{\E\left[H_p(Z_{m+n})s^{Z_{m+n}}\ind_{\Lambda_n}\right]}{\E\left[H_p(Z_{m+n})s^{Z_{m+n}}\right]} & =\lim_{m\to+\infty}\frac{\E\left[H_p(\bar Z_{m+n})(s/\kappa)^{\bar Z_{m+n}}\ind_{\Lambda_n}\right]}{\E\left[H_p(\bar Z_{m+n})(s/\kappa)^{\bar Z_{m+n}}\right]}\\
& =\E\left[\frac{\bar Z_n\bar \kappa^{\bar Z_n-1}}{\bar f'(\bar \kappa)^n}\ind_{\Lambda_n}\right]=\E\left[\frac{\bar Z_n}{\kappa^{\bar Z_n-1}\mu^n}\ind_{\Lambda_n}\right] =\E\left[\frac{Z_n}{\mu^n}\ind_{\Lambda_n}\right].
\end{align*}
\end{proof}

\section{Appendix: A technical lemma on the Hilbert polynomials}

\begin{lem}\label{lem:somme-Hk}
For every $w\ge 1$, for every $k\ge 2$ and every integers $(t_1,\ldots,t_i)$, we have
\[H_w\left(\sum_{j=1}^k t_{j}\right)=\sum_{i=1}^{w\wedge k}\sum_{1\le r_1<\dots<r_i\le k}\sum_{(s_1,\ldots,s_i)\in S_{i,w}}\prod_{j=1}^iH_{s_j}(t_{r_j})\]
\end{lem}

\begin{proof}
We prove this formula by induction on $k$.

First, for $k=2$, the right-hand side of the equation is, for every $w\ge 2$ (the formula is obvious for $w=1$),
\begin{align*}
\sum_{i=1}^{2}\sum_{1\leq r_1<\dots<r_i\leq 2} \sum_{(s_1,\ldots,s_i)\in S_{i,w}}\prod_{j=1}^iH_{s_j}(t_{r_j})
 & =H_w(t_1)+H_w(t_2)+\sum_{s_1=1}^{w-1}H_{s_1}(t_1)H_{w-s_1}(t_2)\\
& =\binom{t_1}{w}+\binom{t_2}{w}+\sum_{s_1=1}^{w-1}\binom{t_1}{s_1}\binom{t_2}{w-s_1}\\
& =\sum_{s_1=0}^{w}\binom{t_1}{s_1}\binom{t_2}{w-s_1}=\binom{t_1+t_2}{w}=H_w(t_1+t_2).
\end{align*}

Assume now that the formula of the lemma is true for every $2\le k$ and let us prove it for $k+1$. We have, using first the formula for $k=2$,
\begin{align*}
H_w\left(\sum_{j=1}^{k+1}t_{j}\right)&=H_w\left(\sum_{j=1}^kt_{j}+t_{k+1}\right)\\
&=H_w\left(\sum_{j=1}^kt_j\right)+H_w(t_{k+1})+\sum_{s=1}^{w-1}H_{s}\left(\sum_{j=1}^kt_j\right)H_{w-s}(t_{k+1})\\
&=\sum_{i=1}^{w\wedge k}\sum_{1\leq r_1<\dots<r_i\leq k}\sum_{(s_1,\ldots,s_i)\in S_{i,w}}\prod_{j=1}^iH_{s_j}(t_{r_j})+\\
& \qquad H_w(t_{k+1})+\sum_{s=1}^{w-1}\sum_{i=1}^{s\wedge k}\sum_{1\leq r_1<\dots< r_i\leq k}\sum_{(s_1,\ldots,s_i)\in S_{i,s}}\left(\prod_{j=1}^iH_{s_j}(t_{r_j})\right)H_{w-s}(t_{k+1})
\end{align*}
by the induction assumption. Inverting the sums in the last term and setting $s_{i+1}=w-s$ than $i'=i+1$ yields
\begin{align*}
H_w\left(\sum_{j=1}^{k+1}t_{j}\right)& =\sum_{i=1}^{w\wedge k}\sum_{1\leq r_1<\dots< r_i\leq k}\sum_{(s_1,\ldots,s_i)\in S_{i,w}}\prod_{j=1}^iH_{s_j}(t_{r_j})+\\
& \qquad H_w(t_{k+1})+\sum_{i=1}^{(w-1)\wedge k}\sum_{1\leq r_1<\dots<r_i\leq k}\sum _{s=i}^{w-1}\sum_{(s_1,\ldots,s_i)\in S_{i,s}}\left(\prod_{j=1}^iH_{s_j}(t_{r_j})\right)H_{w-s}(t_{k+1})\\
& =\sum_{i=1}^{w\wedge k}\sum_{1\leq r_1<\dots<r_i\leq k}\sum_{(s_1,\ldots,s_i)\in S_{i,w}}\prod_{r=1}^iH_{s_j}(t_{r_j})+\\
& \qquad H_w(t_{k+1})+\sum_{i'=2}^{w\wedge (k+1)}\sum_{1\leq r_1<\dots<r_{i'-1}\leq k}\sum_{(s_1,\ldots,s_{i'})\in S_{i',w}}\left(\prod_{j=1}^{i'-1}H_{s_j}(t_{r_j})\right)H_{s_{i'}}(t_{k+1})\\
&=\sum_{i=1}^{w\wedge (k+1)}\sum_{1\leq r_1<\dots<r_i\leq k+1}\sum_{(s_1,\ldots,s_i)\in S_{i,w}}\prod_{j=1}^iH_{s_j}(t_{r_j})
\end{align*}
which is the looked after formula.

\end{proof}

\section*{Acknowledgements}

The authors want to thank Luc Hillairet for several helpful discussions and his indications to obtain Proposition \ref{prop:unicite} or Lemma \ref{lem:formule-a}.

\bibliographystyle{abbrv}
\bibliography{biblio}

\begin{thebibliography}{10}

\bibitem{AD14b}
R.~Abraham and J.-F. Delmas.
\newblock Local limits of conditioned {G}alton-{W}atson trees: the condensation
  case.
\newblock {\em Electron. J. Probab.}, 19:no. 56, 29, 2014.

\bibitem{AD14a}
R.~Abraham and J.-F. Delmas.
\newblock Local limits of conditioned {G}alton-{W}atson trees: the infinite
  spine case.
\newblock {\em Electron. J. Probab.}, 19:no. 2, 19, 2014.

\bibitem{AD17}
R.~Abraham and J.-F. Delmas.
\newblock Asymptotic properties of expanding {G}alton-{W}atson trees.
\newblock arXiv:1712.04650, 2017.

\bibitem{AN}
K.~B. Athreya and P.~E. Ney.
\newblock {\em Branching processes}.
\newblock Dover Publications, Inc., Mineola, NY, 2004.
\newblock Reprint of the 1972 original [Springer, New York; MR0373040].

\bibitem{De09}
P.~Debs.
\newblock Penalisation of the standard random walk by a function of the
  one-sided maximum, of the local time, or of the duration of the excursions.
\newblock In {\em S\'eminaire de {P}robabilit\'es {XLII}}, volume 1979 of {\em
  Lecture Notes in Math.}, pages 331--363. Springer, Berlin, 2009.

\bibitem{De12}
P.~Debs.
\newblock Penalisation of the symmetric random walk by several functions of the
  supremum.
\newblock {\em Markov Process. Related Fields}, 18(4):651--680, 2012.

\bibitem{FW}
K.~Fleischmann and V.~Wachtel.
\newblock On the left tail asymptotics for the limit law of supercritical
  {G}alton-{W}atson processes in the {B}\"ottcher case.
\newblock {\em Ann. Inst. Henri Poincar\'e Probab. Stat.}, 45(1):201--225,
  2009.

\bibitem{HR}
S.~C. Harris and M.~I. Roberts.
\newblock The many-to-few lemma and multiple spines.
\newblock {\em Ann. Inst. Henri Poincar\'e Probab. Stat.}, 53(1):226--242,
  2017.

\bibitem{Ja}
S.~Janson.
\newblock Simply generated trees, conditioned {G}alton-{W}atson trees, random
  allocations and condensation.
\newblock {\em Probab. Surv.}, 9:103--252, 2012.

\bibitem{KMG}
S.~Karlin and J.~McGregor.
\newblock Embedding iterates of analytic functions with two fixed points into
  continuous groups.
\newblock {\em Trans. Amer. Math. Soc.}, 132:137--145, 1968.

\bibitem{Ke}
H.~Kesten.
\newblock Subdiffusive behavior of random walk on a random cluster.
\newblock {\em Ann. Inst. H. Poincar\'e Probab. Statist.}, 22(4):425--487,
  1986.

\bibitem{KNS}
H.~Kesten, P.~E. Ney, and F.~L. Spitzer.
\newblock The galton--watson process with mean one and finite variance.
\newblock {\em Teor. Veroyatnost. i Primenen}, 11(4):579--611, 1966.

\bibitem{RSS}
Y.-X. Ren, S.~R., and S.~S.
\newblock A 2-spine decomposition of the critical {G}alton-{W}atson tree and a
  probabilistic proof of {Y}aglom's theorem.
\newblock ArXiv 1706.07125, 2017.

\bibitem{RVY06b}
B.~Roynette, P.~Vallois, and M.~Yor.
\newblock Limiting laws associated with {B}rownian motion perturbed by its
  maximum, minimum and local time. {II}.
\newblock {\em Studia Sci. Math. Hungar.}, 43(3):295--360, 2006.

\bibitem{RVY06a}
B.~Roynette, P.~Vallois, and M.~Yor.
\newblock Some penalisations of the {W}iener measure.
\newblock {\em Jpn. J. Math.}, 1(1):263--290, 2006.

\bibitem{RVY09}
B.~Roynette, P.~Vallois, and M.~Yor.
\newblock Brownian penalisations related to excursion lengths. {VII}.
\newblock {\em Ann. Inst. Henri Poincar\'e Probab. Stat.}, 45(2):421--452,
  2009.

\bibitem{RY09}
B.~Roynette and M.~Yor.
\newblock {\em Penalising Brownian Paths}.
\newblock 2009.

\bibitem{Se}
E.~Seneta.
\newblock Functional equations and the {G}alton-{W}atson process.
\newblock {\em Advances in Appl. Probability}, 1:1--42, 1969.

\end{thebibliography}

\end{document}